\documentclass[12pt]{amsart}
\usepackage{amsmath}
\usepackage{graphicx}
\usepackage{amsfonts}
\usepackage{amstext}
\usepackage{amsbsy}
\usepackage{amsopn}
\usepackage{amsxtra}
\usepackage{upref}
\usepackage{amsthm}
\usepackage{amssymb}

\newtheorem{thm}{Theorem}[section]
\newtheorem{cor}[thm]{Corollary}
\newtheorem{lem}[thm]{Lemma}
\newtheorem{prop}[thm]{Proposition}
\theoremstyle{definition}
\newtheorem{defn}[thm]{Definition}
\theoremstyle{remark}

\def\D{{\mathbb D}}

\begin{document}

\title {Integral Transforms for Logharmonic Mappings}

\author{H. Arbel\'aez, V. Bravo, R. Hern\'andez, W. Sierra,  \and O. Venegas }
%
\begin{abstract} Bieberbach's conjecture was very important in the development of Geometric Function Theory, not only because of the result itself, but also due to the large amount of methods that have
been developed in search of its proof, it is in this context that the integral transformations of the type $f_\alpha(z)=\int_0^z(f(\zeta)/\zeta)^\alpha d\zeta$ or $F_\alpha(z)=\int_0^z(f'(\zeta))^\alpha d\zeta$  appear. In this notes we extend the classical problem of finding the values of $\alpha\in\mathbb{C}$ for which either $f_\alpha$ or $F_\alpha$ are univalent, whenever $f$ belongs to some subclasses of univalent mappings in $\D$, to the case of logharmonic mappings, by considering the extension of the \textit{shear construction} introduced by Clunie and Sheil-Small in \cite{CSS} to this new scenario.
\end{abstract}

\address{Facultad de Ciencias, Universidad Nacional de Colombia, sede Medell\'in, Colombia.}\email{hjarbela@unal.edu.co}
\address{Facultad de Ingenier\'ia y Ciencias. Universidad Adolfo Ib\'a\~nez\\
Av. Padre Hurtado, Vi\~na del Mar, CHILE.}
\email{rodrigo.hernandez@uai.cl}\email{victor.bravo.g@uai.cl}
\address{Departamento de Matem\'aticas, Universidad del Cauca, Popay\'an, Colombia.}\email{wsierra@unicauca.edu.co}
\address{Departamento de Ciencias Matem\'{a}ticas y F\'{\i}sicas. Facultad de Ingenier\'{\i}a\\ Universidad Cat\'olica de
Temuco.} \email{ovenegas@uct.cl}

\date{\today}
\thanks{The second, third, and fifth author are supported by grant Fondecyt $\#$1190756, Chile.}
\thanks{The first author was supported by the Universidad Nacional de Colombia, Hermes Code 49148.}
\thanks{The fourth author was supported by the Universidad del Cauca through research project VRI ID 5464.}


%

\maketitle
\section{Introduction}

One of the most studied classes of functions in the context of Geometric Function Theory is the well known class $S$ of univalent functions $f$ defined in the unit disk, normalized by $f(0)=1-f'(0)=0.$ The investigations of problems associated to the class $S$ and some of its subclasses go back to the works of Koebe, developed at beginning of the last century. Undoubtedly, one of the most important problems in this field was the Bieberbach conjecture, presented in 1916 and solved by de Brange in 1984. Solving this problem had a profound influence in the development of the theory of univalent functions, providing, among other things, powerful methods which have been used to study other problems in Geometric Function Theory. One important question arising in this context is determine the values of $\alpha\in\mathbb{C}$ for which the functions 
\begin{equation}\label{int-tipo-1}
\varphi_\alpha(z)=\int_0^z\left(\frac{\varphi(\zeta)}{\zeta}\right)^\alpha d\zeta,\end{equation}
or
\begin{equation}\label{int-tipo-2} \Phi_\alpha(z)=\int_0^z(\varphi'(\zeta))^\alpha d\zeta,
\end{equation}
belong to the class $S,$ when $\varphi$ is a function of this class. The univalence of these integral operators was first studied by Royster \cite{Ro65}, and although some partial results have been obtained, in the general case, this question remains open. It is known, for example, that if $|\alpha|\leq 1/4$ then both transformations are univalent, see \cite{km,Pf75}. However, there are functions $\varphi\in S$ for which $\Phi_\alpha$ is not univalent for any $|\alpha|>1/3,$ with $\alpha\neq 1.$ An analogous result was obtained in \cite{km} for the transformation \eqref{int-tipo-1}; in that paper it is proved that for all $|\alpha|>1/2,$ there are functions $\varphi\in S,$ such that the corresponding $\varphi_\alpha$ is not univalent. For a summary of these results we refer the reader to the classical book by Goodman \cite{G}. Some recent results and other problems related to the transformations \eqref{int-tipo-1} and \eqref{int-tipo-2}, in the context of analytic and meromorphic functions, can be found in \cite{Kumar-Sahoo, Kim-Ponnuamy-Sugawa, Nezhmetdinov-Ponnusamy, Ponnusamy-Sahoo-Sugawa}.

Since the work by J. Clunie and T. Sheil-Small \cite{CSS}, many problems  of geometric function theory have been extended, from the setting of holomorphic functions to the  wider class of harmonic mappings in the plane. In this direction, in \cite{VBRHOV1} and subsequently in \cite{ABHSV1}, the authors proposed an extension of the integral transformations \eqref{int-tipo-1} and \eqref{int-tipo-2} to the setting of sense-preserving harmonic mapping. The definitions given in \cite{ABHSV1}  make use of the \textit{shear construction} introduced by J. Clunie and T. Sheil-Small in \cite{CSS} as follows: let $f=h+\overline g$ a sense-preserving harmonic mapping in the unit disk $\mathbb{D}=\left\lbrace z\in\mathbb{C} : |z|<1 \right\rbrace,$ with the usual normalization $g(0)=h(0)=1-h'(0)=0,$ and dilatation $\omega=g'/h'.$ Given $\alpha\in\overline{\mathbb{D}},$ when $\varphi=h-g$ is zero only at $z=0,$ we define $F_\alpha$ as the horizontal shear of $\varphi_\alpha$ defined by \eqref{int-tipo-1} with dilatation $\omega_\alpha=\alpha\omega.$ This is, $F_\alpha=H+\overline{G},$ where $H,G$ satisfy
\[H-G=\varphi_\alpha\qquad\text{and}\qquad \frac{G'}{H'}=\alpha\omega,\]
with $H(0)=G(0)=0.$ In a similar way we extend the integral transform \eqref{int-tipo-2} to sense-preserving harmonic mapping in $\mathbb{D}.$ The reader can find interesting results, some of which are related to the question of knowing for which values of $\alpha,$ the integral transformations lead functions belonging to the class $S$ (or to some of its subclasses) to the class $S$, in \cite{ABHSV1}.

The main objective of the present paper is investigate this type of problems, but in the context of logharmonic mappings defined in a simply connected domain of the complex plane. To this end, we extend the integral operators defined by \eqref{int-tipo-1} and \eqref{int-tipo-2} to the case of logharmonic mappings. Analogously to the extension of \eqref{int-tipo-1} and \eqref{int-tipo-2} previously described for sense-preserving harmonic mapping, we use a method similar to the shear construction by Clunie and Sheil-Small, introduced in \cite{AR,LP} for constructing univalent log-harmonic mappings. The manuscript is organized as follows. In Section\,\ref{sec log map} some preliminaries concerning logharmonic mappings will be introduced, in Section\,\ref{sec I T of firs type} we introduce the integral transform of the first type for logharmonic mappings, in particular, we generalize a result obtained by J. Pfalztgraff in \cite{Pf75}. The extension of the integral transform of the type \eqref{int-tipo-2} to logharmonic mappings is given in Section\,\ref{sec Iint T. 2 type}, in which we obtain similar results to the obtained in the section\,\ref{sec I T of firs type}. Finally, Section\,\ref{sec non vanishing log harmonic map} is devoted to extend the integral transforms to the special case of non-vanishing logharmonic mapping.

\section{Some preliminaries about logharmonic mappings}\label{sec log map}

A logharmonic mapping defined in the unit disk $\mathbb{D}$ is a solution of the nonlinear elliptic partial differential equation \begin{equation*}\label{def-log}
\overline{f_{\overline{z}}(z)}=\omega(z)\left(\dfrac{\overline{f(z)}}{f(z)}\right)f_z(z),\end{equation*} where $\omega$ is an analytic function from $\mathbb{D}$ into itself, which is called second complex dilatation of $f.$ The study of logharmonic functions was initiated mainly by the works of Abdulhadi, Bshouty, and Hengartner \cite{AB88,AH}, and later continued in a series of papers in which is developed the basic theory of logharmonic functions. Since then many papers have been published dealing with this subject, see for example \cite{AE,AM,Aydogan 1,LP,MPW13}.

Because of the condition on $\omega,$ the Jacobian $J_f$ of $f,$ given by
\begin{equation*}\label{jacobiano}
J_f=|f_z|^2-|f_{\overline{z}}|^2=|f_z|^2(1-|\omega|^2),
\end{equation*}
is non negative, and therefore, every non-constant logharmonic mapping is sense preserving and open in $\D.$ If $f$ is a non-constant logharmonic mapping defined in $\D$ and vanishes only at the origin, then $f$ has the representation \begin{equation*}\label{representation-Log}
f(z)=z^m|z|^{2\beta m}h(z)\overline{g(z)},
\end{equation*}
where $m$ is a non-negative integer, $\mathrm{Re}\{\beta\}>-m/2$, and $h$ and $g$ are analytic mappings in the unit disc, such that $g(0)=1$ and $h(0)\neq 0$ (see \cite{AB88}). In particular, when $f$ is a univalent logharmonic mapping defined in $\mathbb{D}$ and $f(0)=0,$ we can represent $f$ in the form
\begin{equation*}\label{representation f univalent}
f(z)=z|z|^{2\beta}h(z)\overline{g(z)},\qquad z\in\mathbb{D},
\end{equation*}
where $\mathrm{Re}\{\beta\}>-1/2$, and $h$ and $g$ are analytic mappings in $\mathbb{D},$ such that $g(0)=1$ and $0\notin hg(\mathbb{D}).$ This class of logharmonic mappings has been widely studied. For more details the reader can read the summary paper in \cite{AR} and the references given there.

In the first part of this study, we consider sense-preserving univalent log-harmonic mappings in $\mathbb{D}$ with $\omega(0)=0,$ case in which $f$ has the form
\[f(z)=zh(z)\overline{g(z)},\qquad z\in\mathbb{D},\]
which implies that its dilatation $\omega$ is given by $$\omega(z)=\frac{zg'(z)/g(z)}{1+zh'(z)/h(z)}.$$ 
In relation to this class of functions, denoted by $S_{Lh},$ the following result, which is one of the tools that we shall use in this manuscript, was proved in \cite{AM}, see also \cite{AH}; it asserts that\\

\noindent \textbf{Theorem A}\textit{ Let $f(z)=zh(z)\overline{g(z)}$ be a logharmonic mapping defined in $\D$ such that $0\notin hg(\mathbb{D})$. Then $\varphi(z)=zh(z)/g(z)$ is a starlike analytic function if and only if $f(z)$ is a starlike logharmonic mapping.}\\

Also we will consider non-vanishing logharmonic mappings in $\mathbb{D}.$ It is well known that such mappings can be expressed in the form
\begin{equation*}
f(z) = h(z)\overline{g(z)},
\end{equation*}
where $h$ and $g$ are non-vanishing analytic functions in $\D.$ In terms of $h$ and $g,$ when $f$ is locally univalent, the dilatation $\omega$ of $f$ is given by
\begin{equation*}\label{omega}
\omega=\frac{g'h}{gh'}.
\end{equation*}
Note that in this case, $h$ is locally univalent.\\
 
Trough this paper, we shall mainly study integral transformations of the type \eqref{int-tipo-1} and \eqref{int-tipo-2} for two class of logharmonic functions:  first we will consider univalent logharmonic mappings of the form $f(z)=zh(z)\overline{g(z)}$ normalized as $h(0)=g(0)=1$ and its several subfamilies as starlike and convex mappings normalized as before. In a similar way we will study the case when $f$ is a non-vanishing logharmonic function of the form $f=h\overline{g}$ such that $h(0)=g(0)=1.$

\section{Integral Transform of the first type}\label{sec I T of firs type}

Let $f=zh(z)\overline{g(z)}$ be a locally univalent logharmonic mapping defined in the unit disk with the normalization described above. Then $\varphi(z)=zh(z)/g(z)$ is an analytic function in $\mathbb{D}$ such that $\varphi(0)=0,$ $\varphi'(0)=1,$ and $\varphi(z)\neq 0$ if $z\neq 0.$ So, $h,g$ are solution of the system of nonlinear differential equations
\[\frac{zh(z)}{g(z)}=\varphi(z)\qquad\text{and} \qquad \frac{zg'(z)/g(z)}{1+zh'(z)/h(z)}=\omega(z),\quad z\in\mathbb{D}.\]
This fact was used in \cite{LP} to establish a method of construction of logharmonic mappings satisfying certain properties, allows us extend the integral transformation \eqref{int-tipo-1} to the case of logharmonic mappings, with the above notation, as follows: given $\alpha\in\overline{\mathbb{D}},$ we define $f_\alpha$ be the logharmonic mapping given by $f_\alpha(z)=zH(z)\overline{G(z)},$ where $H,G$ satisfy the system
\begin{equation}\label{f alpha tipo 1} \frac{zH(z)}{G(z)}=\varphi_\alpha(z)=\int_{0}^{z}\left(\frac{\varphi(\xi)}{\xi}\right)^\alpha d\xi\quad\text{and}\quad \frac{zG'(z)/G(z)}{1+zH'(z)/H(z)}=\omega_\alpha(z):=\alpha\omega(z),\end{equation}
with the initial conditions $H(0)=G(0)=1.$

Note that if $\alpha=0,$ then $H/G\equiv 1$ and $\omega_\alpha=0$, whence it follows that $G$ is constant, say $k$, which leads to $f_\alpha(z)=|k|^2z$. On the other hand, if $\alpha=1$ we have that $\varphi_1(z)=zH(z)/G(z)$ 
satisfies
\[1+z\frac{\varphi_1''(z)}{\varphi_1'(z)}=z\frac{\varphi'(z)}{\varphi(z)}\,,\]
which implies that when $\varphi$ is starlike, then $\varphi_1$ is convex (in particular starlike). Consequently, according to \cite{LP}, $f_1$ is starlike. However, we know from the analytical case that for $\alpha$ of module greater than $1/3,$ there are univalent functions whose respective functions $f_\alpha$ are not univalent in $\D$. For this and other similar problems that appear both in the context of analytic functions and in the context of harmonic mapping, we are interested in studying the values of $\alpha$ for which we can ensure that $f_\alpha$ belongs to $S_{Lh}$ or to some of its subclasses. A first result in this direction is given in the following proposition.

\begin{prop}\label{prop f_alpha is starlike} Let $f(z)=zh(z)\overline{g(z)}$ be a logharmonic mapping defined in $\D$ with dilatation $\omega$. If $zh(z)/g(z)=\varphi(z)$ is a convex mapping, then $f_\alpha$ defined by (\ref{f alpha tipo 1}) is a starlike logharmonic mapping in the unit disk for $\alpha\in[0,1]$.
\end{prop}
\begin{proof} We note first that if $f_z(z_0)=0$ for some $z_0\in\mathbb{D},$ then $z_0\neq 0,$ 
\[z_0h'(z_0)+h(z_0)=0 \qquad \text{and}\qquad g'(z_0)=0,\]
which implies $\varphi'(z_0)=0.$ Hence, $J_f$ is positive in $\mathbb{D}.$

By definition of $f_\alpha(z)=zH(z)\overline{G(z)}$ we have that $zH(z)/G(z)=\varphi_\alpha(z)$ satisfies $$1+\mbox{Re}\left\{z\frac{\varphi''_\alpha (z)}{\varphi'_\alpha (z)}\right\}=1+\alpha\mbox{Re}\left\{z\frac{\varphi'(z)}{\varphi(z)}-1\right\},$$
$\alpha\in\mathbb{R}.$ Since $\varphi$ is convex, it follows that $\textrm{Re}\{z\varphi'(z)/\varphi(z)\}>1/2$ which implies that $$1+\mbox{Re}\left\{z\frac{\varphi''_\alpha (z)}{\varphi'_\alpha(z)}\right\}> 1-\frac{\alpha}{2}>0,$$
for $\alpha\in[0,1].$ Hence, $\varphi_\alpha$ is convex and, in particular, is a starlike mapping. Applying Theorem\,A the proof is concluded. 
\end{proof}
We remark that the conclusion of the proposition remains true if $\alpha\in [0,\rho],$ where $\rho=\min\{2,1/\|\omega\|\},$ and $\|\omega\|=\sup\{ |\omega(z)| : z\in\mathbb{D}\}.$\\ 

In \cite{AE} the authors extend the concept of stable univalent function (resp. starlike, convex, close-to-convex, etc.) studied in \cite{rhmje} for harmonic and analytic functions, to the case of logharmonic mappings. More precisely, they give the
following definition.
\begin{defn}
A logharmonic function $f(z)=zh(z)\overline{g(z)}$, normalized by $h(0)=g(0)=1$, is called stable univalent logharmonic or $SU_{Lh}$, if for all $|\lambda|=1$, $f_{\lambda}(z)=zh(z)\overline{g(z)}^{\lambda}$ is univalent logharmonic.
\end{defn}
Next, we present a criterion to establish the stable univalence of $f_\alpha.$
\begin{thm}\label{Theorem-univ.}
Let $f(z)=zh(z)\overline{g(z)}$ be a logharmonic mapping defined in the unit disc with dilatation $\omega$, normalized by $h(0)=g(0)=1$, and such that $zh(z)/g(z)=\varphi(z)$ is a univalent mapping. Then $f_\alpha$ defined by (\ref{f alpha tipo 1}) is a stable univalent logharmonic mapping if $\alpha$ satisfies the inequality
\begin{equation}\label{cond-thm-1}
|\alpha|\leq\dfrac{1}{4\left(2\delta+1+8\|\omega^*\|/15\right)},
\end{equation} 
where
\[\delta=\frac{4\|\omega\|}{4-\|\omega\|}\qquad \text{and}\qquad \|\omega^\ast\|=\sup_{z\in\D}\frac{|\omega'(z)|(1-|z|^2)}{1-|\omega(z)|^2}.\]
\end{thm}
\begin{proof} 
As in the previous proposition, $J_f>0$ in $\mathbb{D}$ and therefore $f$ is locally univalent. For $|\lambda|=1$ we define 
\begin{equation*}\label{Psi}
\psi_{\lambda}(z)=z\frac{H(z)}{G(z)^{\lambda}}\,, \qquad\qquad z\in\mathbb{D},
\end{equation*}
from where
\[\psi'_\lambda(z)= \psi_\lambda(z)\left(\frac{1}{z}+\frac{H'(z)}{H(z)}-\lambda\frac{G'(z)}{G(z)}\right)=\psi_\lambda(z)\frac{1}{z}\left(1+\frac{zH'(z)}{H(z)}\right)(1-\lambda\omega_{\alpha}(z)).\]
On the other hand, 
\begin{equation*}
\frac{z\varphi'_{\alpha}(z)}{\varphi_\alpha(z)}= 1+z\frac{H'(z)}{H(z)}-z\frac{G'(z)}{G(z)}=\left(1+z\frac{H'(z)}{H(z)}\right)(1-\omega_{\alpha}(z)),   
\end{equation*}
from which we get
\begin{equation}\label{eq3}
\psi'_\lambda(z)=\psi_\lambda(z)\frac{1-\lambda\omega_\alpha(z)}{1-\omega_\alpha(z)}\frac{\varphi'_\alpha(z)}{\varphi_{\alpha}(z)},
\end{equation}
and
\begin{equation*}
\frac{\psi''_\lambda(z)}{\psi'_\lambda(z)}=\frac{\psi'_\lambda(z)}{\psi_\lambda(z)}+\frac{\varphi''_\alpha(z)}{\varphi'_{\alpha}(z)}-\frac{\varphi'_\alpha(z)}{\varphi_{\alpha}(z)}
+\frac{(1-\lambda)\omega'_\alpha(z)}{(1-\lambda\omega_\alpha(z))(1-\omega_\alpha(z))}.  
\end{equation*}
Replacing (\ref{eq3}), in the last equality, we get
\begin{equation}\label{eq4}
 \frac{\psi''_\lambda(z)}{\psi'_\lambda(z)}=\frac{(1-\lambda)\omega_\alpha(z)}{1-\omega_\alpha(z)}\frac{\varphi'_\alpha(z)}{\varphi_{\alpha}(z)}+\frac{\varphi''_\alpha(z)}{\varphi'_{\alpha}(z)}
+\frac{(1-\lambda)\omega'_\alpha(z)}{(1-\lambda\omega_\alpha(z))(1-\omega_\alpha(z))},
\end{equation} 
and therefore
\begin{align*}
(1-|z|^2)\left|z\frac{\psi''_\lambda(z)}{\psi'_\lambda(z)}\right|&=(1-|z|^2)\left|\frac{(1-\lambda)\omega_\alpha(z)}{1-\omega_\alpha(z)}\frac{z\varphi'_\alpha(z)}{\varphi_{\alpha}(z)}+\frac{z\varphi''_\alpha(z)}{\varphi'_{\alpha}(z)}+\frac{z(1-\lambda)\omega'_\alpha(z)}{(1-\lambda\omega_\alpha(z))(1-\omega_\alpha(z))}\right|\\[0.3cm]
&\hspace{-1.5cm} = (1-|z|^2)|\alpha|\left|\frac{(1-\lambda)\omega(z)}{1-\omega_\alpha(z)}\frac{z\varphi'_\alpha(z)}{\varphi_{\alpha}(z)}+\frac{z\varphi'(z)}{\varphi(z)}-1
+\frac{z(1-\lambda)\omega'(z)}{(1-\lambda\omega_\alpha(z))(1-\omega_\alpha(z))}\right|\\[0.3cm]
&\hspace{-1.5cm} \leq (1-|z|^2)|\alpha|\left(\left|\frac{(1-\lambda)\omega(z)}{1-\omega_\alpha(z)}\frac{z\varphi'_\alpha(z)}{\varphi_{\alpha}(z)}\right|+\left|\frac{z\varphi'(z)}{\varphi(z)}-1\right|+\left|\frac{z(1-\lambda)\omega'(z)}{(1-\lambda\omega_\alpha(z))(1-\omega_\alpha(z))}\right|\right) .
\end{align*}
Since $|\alpha|\leq 1/4$ then $\varphi_\alpha\in S$, therefore
\begin{equation} \label{Theorem univ 2}
\left|\frac{z\varphi'_\alpha(z)}{\varphi_{\alpha}(z)}\right|\leq \frac{1+|z|}{1-|z|} \qquad\qquad \text{and}  \qquad\qquad \left|\frac{z\varphi'(z)}{\varphi(z)}-1\right|\leq \frac{2}{1-|z|},
\end{equation}
being the second inequality a consequence of $\varphi \in S$. Moreover, it is easy to see that
\begin{equation}\label{ineq-omega}
  \frac{|\omega(z)|}{|1-\omega_\alpha(z)|}\leq \frac{4\|\omega\|}{4-\|\omega\|}, 
\end{equation}
and
\begin{equation}\label{ineq-omega-ast}
\frac{|\omega'(z)|(1-|z|^2)}{|(1-\lambda\omega_\alpha(z))(1-\omega_\alpha(z))|}\leq \frac{\|\omega^\ast\|(1-|\omega(z)|^2)}{(1-|\alpha||\omega(z)|)^2}\leq \frac{\|\omega^\ast\|(1-|\omega(z)|^2)}{(1-|\omega(z)|/4)^2}\leq \frac{16}{15}\|\omega^\ast\|,
\end{equation}
for all $z\in\mathbb{D}.$ Thus, using inequalities (\ref{Theorem univ 2}), (\ref{ineq-omega}), and (\ref{ineq-omega-ast}) it follows that
\begin{equation*}
 (1-|z|^2)\left|z\frac{\psi''_\lambda(z)}{\psi'_\lambda(z)}\right|\leq 2|\alpha|\left(\frac{16\|\omega\|}{4-\|\omega\|}+2+\frac{16}{15}\|\omega^*\|\right).
\end{equation*}
Inequality (\ref{cond-thm-1}) and Becker's criterion(see \cite{B72})  imply that $\psi_\lambda$ is univalent for all $\lambda$ in the unit circle, which complete the proof.
\end{proof}
 Note that in the case when $f$ is an analytic mapping, $\omega \equiv 0$ and consequently $\delta=0$ and $\|\omega^\ast\|=0.$ Then the inequality (\ref{cond-thm-1}) becomes in $|\alpha|\leq 1/4,$ which is the best known bound of the values of $\alpha$ such that the corresponding $f_\alpha$ is univalent in $\D$.\\
 
The authors proved in \cite[Lemma 3]{ABHSV1} that if $\varphi$ is a univalent analytic function in a linear invariant family $\mathcal F$ with order $\beta$, then \begin{equation}\label{family orden beta} (1-|z|^2)\left|z\frac{\varphi'(z)}{\varphi(z)}\right|\leq 2\beta, \quad \text{for all} \quad z\in \D.
\end{equation}
By using this inequality we can improve the last theorem.
\begin{cor} 
Let $f, \varphi$ be as in Theorem\,\ref{Theorem-univ.}. Suppose moreover that  $\varphi$ is univalent in a linear invariant family $\mathcal F$ with order $\beta.$ Then $f_\alpha$ is stable univalent logharmonic if $\alpha$ satisfies the condition 
$$|\alpha|\leq \min\left\lbrace \frac{1}{2(4\delta+\beta + \frac{1}{2}+16\|\omega^*\|/15)}\,, \,  \frac{1}{4}\right\rbrace.$$
\end{cor}
\begin{proof}
The proof is completely analogous to that of the previous theorem; it is enough to replace the second condition in \eqref{Theorem univ 2} by the inequality $$\left|\frac{z\varphi'(z)}{\varphi(z)}-1\right|\leq \frac{2\beta}{1-|z|^2}+1,$$
which is an immediate consequence of \eqref{family orden beta}.
\end{proof}
\section{Integral Transform of the second type}\label{sec Iint T. 2 type}
Following the same procedure as in the previous section, we extend the integral transformation \eqref{int-tipo-2} to the case of logharmonic mappings as follows: given $\alpha\in\overline{\mathbb{D}}$ and $f(z)=zh(z)g(z)$ a locally univalent logharmonic mapping defined in the unit disk, with the normalization $h(0)=g(0)=1$ and $0\notin hg(\mathbb{D}),$ we define $F_\alpha$ be the logharmonic mapping, with dilatation $\omega_\alpha=\alpha\omega,$ given by $F_\alpha=zH(z)G(z),$ where $H,G$ satisfy the system
\begin{equation}\label{second type} z\frac{H(z)}{G(z)}=\Phi_\alpha(z)=\int_{0}^{z}(\varphi'(\xi))^\alpha d\xi \qquad\text{and}\qquad \omega_\alpha=\alpha\omega,
\end{equation}
with the initial conditions $H(0)=G(0)=1.$\\

The following two results are dual to Proposition\,\ref{prop f_alpha is starlike} and Theorem\,\ref{Theorem-univ.}, respectively. In both cases we will assume the above normalization.

\begin{prop}Let $f(z)=zh(z)\overline{g(z)}$ be a logharmonic mapping defined in $\D$ with dilatation $\omega$. If $zh(z)/g(z)=\varphi(z)$ is a convex mapping, then $F_\alpha$ defined by \eqref{second type} is a starlike logharmonic mapping in the unit disk for $\alpha\in[0,1]$.
\end{prop}
\begin{proof} As in Proposition\,\ref{prop f_alpha is starlike}, the condition on $\varphi$ implies that $J_f$ is positive, so $f$ is locally univalent. Since $\varphi$ is a convex mapping, we have that for $\alpha\in[0,1],$ \[\textrm{Re}\left\{1+z\frac{\Phi_\alpha''(z)}{\Phi_\alpha'(z)}\right\}=1+\alpha\textrm{Re}\left\{z\frac{\varphi''(z)}{\varphi'(z)}\right\}> 1-\alpha>0,\]
from which it concludes that $\Phi_\alpha(z)$ is a convex mapping and in particular starlike. Using Theorem\,A, we have that $F_\alpha$ is a starlike logharmonic mapping.
\end{proof}

\begin{thm} Let $f(z)=zh(z)\overline{g(z)}$ be a logharmonic mapping defined in $\mathbb{D}$ with dilatation $\omega$ and such that $zh(z)/g(z)=\varphi(z)$ is a univalent mapping. Then $F_\alpha$ defined as in \eqref{second type} is a stable univalent logharmonic mapping if $\alpha$ satisfies the inequality 
\begin{equation*}\label{cond-thm-2}
|\alpha|\leq\dfrac{1}{2(4\delta+3+16\|\omega^*\|/15)},
\end{equation*}
with $\delta$ and $\|\omega^*\|$ as in Theorem\,\ref{Theorem-univ.}.
\end{thm}
\begin{proof} Since $\varphi$ is univalent we have $J_f>0$ and
\begin{equation}\label{eq15}\sup_{z\in\D}(1-|z|^2)\left|\frac{\Phi''_\alpha(z)}{\Phi'_\alpha(z)} \right|\leq |\alpha|\sup_{z\in\D}(1-|z|^2)\left|\frac{\varphi''(z)}{\varphi'(z)} \right|\leq 6|\alpha|.
\end{equation}
It follows from equation (\ref{eq4}) and inequalities (\ref{Theorem univ 2}), (\ref{ineq-omega}), and (\ref{ineq-omega-ast}) that $$(1-|z|^2)\left|z\frac{\psi''_\lambda(z)}{\psi'_\lambda(z)}\right|\leq2|\alpha|\left(4\delta+3+16\|\omega^\ast\|/15\right),$$ from which we get, in virtue of the classical univalence criterion of Becker, that $F_\alpha$ is a stable univalent logharmonic mapping.
\end{proof} 
Note that if $\varphi$ is a convex mapping, then the right side of the inequality (\ref{eq15}) can be replaced by $4|\alpha|$, from where the range of the values of $\alpha$ for which the corresponding mapping $F_\alpha$ is stable univalent, is determined by the inequality $$|\alpha|\leq \frac{1}{4(2\delta+1+8\|\omega
^*\|/15)}.$$
\section{Alternative definition of non vanishing Integral Transform}\label{sec non vanishing log harmonic map}

In this section, we extend the integral transformations \eqref{int-tipo-1} and \eqref{int-tipo-2} to non-vanishing logharmonic mappings in $\mathbb{D}.$ To this end, we will use the fact that there is a branch of $\log f,$ which is harmonic in $\mathbb{D}$ with dilatation $\omega_{\log f} = \omega_f,$ and apply to it the theory that we developed in \cite{ABHSV1}. Note that $\omega_{\log f} = \omega_f$ implies that $\log f$ is a sense-preserving harmonic mapping and therefore $J_f$ is positive in $\mathbb{D}.$

As was mentioned in Subsection\,\ref{sec log map}, if $f$ is a non-vanishing logharmonic mapping in $\mathbb{D}$ with dilatation $\omega,$ then $f=h\overline{g},$ where $h,g$ are non-vanishing analytic functions in $\mathbb{D}.$ In this case $\omega=g'h/gh',$ and if we assume the normalization $h(0)=g(0)=1,$ we can choose branches of $\log f,$ $\log h,$ and $\log g,$ satisfying
\[\log f(0)=\log h(0)=\log g(0)=0\qquad\text{and}\qquad \log f=\log h+\overline{\log g}.\]
The following lemma is a general tool that appears as a natural version of Theorem\,1 in \cite{CH} for logharmonic mappings. We will use this lemma in the next subsections.

\begin{lem}\label{non-vanishing}
Let $f=h\overline{g}$ be a non-vanishing logharmonic mapping in $\mathbb{D}$ with dilatation $\omega,$ and suppose that $\psi=\log h/g$ is a univalent mapping such that $\Omega:=\psi(\mathbb{D})$ is $M-$linearly connected. Then $f$ is univalent in $\D$ if $\|\omega\|<1/(2M+1)$.
\end{lem}
\begin{proof}
Suppose there are $z_1\neq z_2$ in $\D$ such that $f(z_1)=f(z_2),$ and let $S$ be a path in $\Omega$ joining $\psi(z_1)$ to $\psi(z_2),$ such that $\ell(S)\leq M |\psi(z_1)-\psi(z_2)|.$ So,
$$e^{\psi(z_1)}|g(z_1)|^2=e^{\psi(z_2)}|g(z_2)|^2,$$
and in consequence,
\begin{equation*}\label{eq}
|\psi(z_1)-\psi(z_2)| \leq |2\log g(z_1)-2\log g(z_2)| \leq  2\int\limits_\gamma \left|\dfrac{g'(\xi)}{g(\xi)}\right|\, |d\xi|,
\end{equation*}
where $\gamma=\psi^{-1}(S).$ From here and the equality $$\frac{g'}{g}=\omega\frac{h'}{h}=\frac{\omega}{1-\omega}\psi',$$
it follows that $$|\psi(z_1)-\psi(z_2)| \leq 2\frac{\|\omega\|}{1-\|\omega\|}\int\limits_\gamma |\psi'(\xi)| |d\xi|\leq 2\frac{\|\omega\|}{1-\|\omega\|}l(S)<|\psi(z_1)-\psi(z_2)|,$$
if $\omega$ satisfies $\|\omega\|<1/(2M+1).$ This contradiction ends the proof.
\end{proof} 

\subsection{Integral Transform of the first type for non-vanishing logharmonic mappings.} We consider a non-vanishing logharmonic mapping $f=h\overline g$ defined in $\mathbb{D},$ with dilatation $\omega=g'h/h'g,$ and normalized by $h(0)=g(0)=1.$ We suppose moreover that $\varphi=\log h-\log g$ is zero only at $z=0.$ Given $\alpha\in\overline{\mathbb{D}},$ we define the integral transform of the first type of $f$ by
\begin{equation}\label{alt-def-first-type} f_\alpha=e^H\,\overline{e^G}=\exp\{H+\overline{G}\},
\end{equation}
where $H$ and $G$ satisfy the system
\[H(z)-G(z)=\varphi_\alpha(z)=\int_0^z \left(\frac{\varphi(\zeta)}{\zeta}\right)^\alpha d\zeta\qquad\text{and}\qquad \frac{G'}{H'}=\alpha\omega,\]
with the initial conditions $H(0)=G(0)=0.$ In other words, $f_\alpha$ is deﬁned in such way that $\log f_\alpha=H+\overline G$ is a harmonic branch of the logarithm of $f_\alpha$ in $\mathbb{D},$ which is the horizontal shear of $\varphi_\alpha$ with dilatation $\alpha\omega$. The reader can find the details of the shear construction of harmonic mappings in  \cite{CSS}.

\begin{prop} Let $f=h\overline{g}$ be a non-vanishing logharmonic mapping in $\mathbb{D}$ with dilatation $\omega$  and let $f_\alpha$ be defined by equation (\ref{alt-def-first-type}).
\begin{enumerate}
    \item[$(i)$] If $\varphi=\log h/g$ is a convex function and $\alpha\in[0,0.745]$, then $f_\alpha$ is a univalent logharmonic mapping in $\D$.\\
    \item[$(ii)$] If $\varphi=\log h/g$ is a starlike function and $\alpha\in[0,0.6]$, then $f_\alpha$ is a univalent logharmonic mapping in $\D$.
\end{enumerate}    
\end{prop}
\begin{proof}
We have that $\log f_\alpha=H+\overline{G}.$ So, in virtue of Theorem\,1 in \cite{CH}, it is sufficient to prove, in each case, that $H$ is convex. To prove $(i)$ we see that $H$ satisfies $$1+z\frac{H''(z)}{H'(z)}=1-\alpha+\alpha z\frac{\varphi'(z)}{\varphi(z)}+\alpha \frac{z\omega'(z)}{1-\alpha\omega(z)},$$
whence, since $\alpha$ is a real number, we get \begin{equation}\label{prop2}\mbox{Re}\left\{1+z\frac{H''(z)}{H'(z)}\right\}\geq 1-\alpha+\alpha \mbox{Re}\left\{z\frac{\varphi'(z)}{\varphi(z)}\right\}-\alpha\arcsin{\alpha}.
\end{equation} 
Therefore, from $\textrm{Re}\{z\varphi'/\varphi(z)\}\geq 1/2$ and inequality (\ref{prop2}), we conclude that $$\mbox{Re}\left\{1+z\frac{H''(z)}{H'(z)}\right\}\geq 1-\frac{\alpha}{2}-\alpha\arcsin{\alpha}> 0,$$ if $\alpha\in[0,0.745]$. Now, to prove $(ii)$ we use $\textrm{Re}\{z\varphi'/\varphi(z)\}\geq 0$ and inequality (\ref{prop2}) to obtain $$\textrm{Re}\left\{1+z\frac{H''(z)}{H'(z)}\right\}\geq 1-\alpha-\alpha\arcsin{\alpha}> 0,$$
when $\alpha\in[0,0.6]$. 
\end{proof}
\begin{prop}\label{prop 4.3}
 Let $f=h\overline{g}$ be a non-vanishing logharmonic mapping defined in $\mathbb{D}$ with $\|\omega\|<1/3$ and let $f_\alpha$ be defined by equation (\ref{alt-def-first-type}). If $\varphi=\log h/g$ is a convex function and $\alpha\in [0,2]$, then $f_\alpha$ is univalent.
 \end{prop}
\begin{proof}
The proof follows as a direct application of Lemma\,\ref{non-vanishing} and the fact that $\varphi_\alpha$ is a convex mapping for $\alpha\in [0,2],$ case in which $\varphi_\alpha(\D)$  is a $M$-linearly connected domain with $M=1$.
\end{proof}
\begin{thm} \label{Theorem-nonvanishing-first type}
 Let $f=h\overline{g}$ be a non-vanishing logharmonic mapping in $\mathbb{D}$ with dilatation $\omega$ and let $f_\alpha$ be defined by equation (\ref{alt-def-first-type}). If $\varphi=\log h/g$ is a univalent function and $|\alpha|\leq 0.165,$ then $f_\alpha$ is univalent.
\end{thm}
\begin{proof}
For $|\lambda|=1$ we define $\Psi_\lambda=H+\lambda G.$ A direct calculation shows that
\begin{equation*}
    \frac{H''(z)}{H'(z)}=\alpha\left(\frac{\varphi'(z)}{\varphi(z)}-\frac{1}{z}\right)+\frac{\omega'_\alpha(z)}{1-\omega_\alpha(z)}
\end{equation*}
and
\begin{equation*}
\frac{z\Psi''_\lambda(z)}{\Psi'_\lambda(z)}=\frac{zH''(z)}{H'(z)}+\frac{\lambda z\omega'_\alpha(z)}{1+\lambda\omega_\alpha(z)}=\alpha\left(\frac{z\varphi'(z)}{\varphi(z)}-1\right)+\frac{(1+\lambda)z\omega'_\alpha(z)}{(1-\omega_\alpha(z))(1+\lambda\omega_\alpha(z))},
\end{equation*}
from which we see that
\begin{align*}
 (1-|z|^2)\left|\frac{z\Psi''_\lambda(z)}{\Psi'_\lambda(z)}\right|&\leq |\alpha|\left[(1-|z|^2)\left |\frac{z\varphi'(z)}{\varphi(z)}-1\right|+2\frac{(1-|z|^2)|\omega'(z)|}{(1-|\alpha||\omega(z)|)^2}\right]\\
 & \leq |\alpha|\left[4+2\frac{(1-|\omega(z)|^2)\left\|\omega^*\right\|}{(1-|\alpha||\omega(z)|)^2}\right]\\
 &\leq 2|\alpha|\left(2+\frac{1}{1-|\alpha|^2}\right),
\end{align*}
being the last inequality a consequence of
$$\max_{z\in \D}\frac{(1-|\omega(z)|^2)}{(1-|\alpha||\omega(z)|)^2}\leq\frac{1}{1-|\alpha|^2}\qquad\text{and}\qquad \left\|\omega^*\right\|\leq 1.$$
It follows from the Becker’s  criterion that $\Psi_\lambda$ is univalent if  $2|\alpha|\left(2+\frac{1}{1-|\alpha|^2}\right)\leq 1$, whence $H+\overline{G}$ is stable harmonic univalent if $|\alpha|\leq 0.165.$ In consequence, $f_\alpha$ is univalent for these values of $\alpha.$
\end{proof}
\subsection{Integral Transform of the second type for non-vanishing logharmonic mappings.} The definition of the integral transform of the second type for non-vanishing logharmonic mappings, is completely analogous to that given in the previous subsection: let $f=h\overline{g}$ be a non-vanishing logharmonic mapping in $\mathbb{D}$ with dilatation $\omega=g'h/h'g$ and normalized by $h(0)=g(0)=1.$ Note that from the condition $|\omega(z)|<1$ for all $z\in\mathbb{D},$ it follows that $\varphi=\log h-\log g$ is locally univalent in $\mathbb{D}.$ We define the logharmonic mapping $F_\alpha=e^H\overline{e^G},$ where $H,G$ satisfy the system \begin{equation}\label{Phi}
H(z)-G(z)=\Phi_\alpha(z)=\int_0^z\left(\varphi'(\zeta)\right)^\alpha d\zeta \qquad\text{and}\qquad\omega_{F_\alpha}=\alpha\omega,
\end{equation}
with the initial conditions $H(0)=G(0)=0.$
\begin{prop} Let $f=h\overline{g}$ be a non-vanishing logharmonic mapping defined in $\mathbb{D}$ with dilatation $\omega$ and let $F_\alpha$ be defined by equation (\ref{Phi}). If $\varphi$ is a convex function and $\alpha\in[0,0.605]$, then $F_\alpha$ is a univalent logharmonic mapping in $\D$.
\end{prop}
\begin{proof}A straightforward calculation leads us to 
$$1+z\frac{H''(z)}{H'(z)}=1+\alpha z\frac{\varphi''(z)}{\varphi'(z)}+\alpha \frac{z\omega'(z)}{1-\alpha\omega(z)},$$
which implies, together with the convexity of $\varphi,$ that
\begin{align*}
\textrm{Re}\left\{1+z\frac{H''(z)}{H'(z)}\right\}&= 1-\alpha+\alpha \textrm{Re}\left\{1+z\frac{\varphi''(z)}{\varphi'(z)}\right\}+\alpha\textrm{Re}\left\{ \frac{z\omega'(z)}{1-\alpha\omega(z)}\right\}\\
&\geq 1-\alpha - \alpha\arcsin{\alpha}.
\end{align*}
Hence, $H$ is convex, and therefore $\log F_\alpha$ is univalent, if $|\alpha|\leq 0.605$.
\end{proof}
The proof of the following proposition is essentially the same as that of Proposition\,\ref{prop 4.3}; so we omit its proof.
\begin{prop}
 Let $f=h\overline{g}$ be a non vanishing logharmonic mapping defined in the unit disk with $\|\omega\|<1/3$ and let $F_\alpha$ be defined by equation (\ref{Phi}). If $\varphi$ is a convex function, then  $F_\alpha$ is  univalent for $\alpha \in [0,1]$.
\end{prop}
\begin{thm}
 Let $f=h\overline{g}$ be a non-vanishing logharmonic mapping in $\mathbb{D}$ with dilatation $\omega$ and let $F_\alpha$ be defined by equation (\ref{Phi}). If $\varphi$ is a univalent function and $|\alpha|\leq 0.125$, then $F_\alpha$ is univalent.
\end{thm}
\begin{proof}
For $|\lambda|=1$ we define $\Psi_\lambda=H+\lambda G.$ Using (\ref{Phi}), we obtain by a direct calculation
\begin{equation*}
    \frac{H''(z)}{H'(z)}=\alpha\left(\frac{\varphi'(z)}{\varphi(z)}-\frac{1}{z}\right)+\frac{\omega'_\alpha(z)}{1-\omega_\alpha(z)}
\end{equation*}
and
\begin{equation*}
\frac{z\Psi''_\lambda(z)}{\Psi'_\lambda(z)}=\frac{zH''(z)}{H'(z)}+\frac{\lambda z\omega'_\alpha(z)}{1+\lambda\omega_\alpha(z)}=\alpha\left(\frac{z\varphi'(z)}{\varphi(z)}-1\right)+\frac{(1+\lambda)z\omega'_\alpha(z)}{(1-\omega_\alpha(z))(1+\lambda\omega_\alpha(z))}.
\end{equation*}
It follows from the univalence of $\varphi$ and 
$$\max_{z\in \D} \frac{(1-|\omega(z)|^2)}{(1-|\alpha||\omega(z)|)^2}\leq\frac{1}{1-|\alpha|^2},$$
that
\begin{align*}
(1-|z|^2)\left|\frac{z\Phi''_\lambda(z)}{\Phi'_\lambda(z)}\right|&\leq |\alpha|\left(\left|(1-|z|^2)\frac{\varphi''(z)}{\varphi'(z)}-2\overline{z}\right|+2|z|^2+2\frac{(1-|\omega(z)|^2)\left\|\omega^*\right\|}{(1-|\alpha||\omega(z)|)^2}\right)\\
&\leq 2|\alpha|\left(3+\frac{1}{1-|\alpha|^2}\right).
\end{align*}
Consequently, we conclude from the Becker's criterion that $\Psi_\lambda$ is univalent if $|\alpha|\leq 0.125,$ and therefore $\log F_\alpha$ is stable harmonic univalent for these values of $\alpha.$ This complete the proof of the theorem.
\end{proof}

\end{document}